\documentclass[12pt]{amsart}
\usepackage[all,cmtip]{xy}
\usepackage[all]{xy}
\usepackage{amssymb}
\usepackage{amssymb}
\usepackage{amsmath}
\usepackage{xcolor}
\usepackage{algorithmicx}
\usepackage[english]{babel}
\usepackage{graphics}
\usepackage{graphicx}
\graphicspath{ {images/} }
\usepackage{epstopdf}

\usepackage{amsmath}
\usepackage{mathrsfs}
\usepackage{tikz-cd}
\usetikzlibrary{shapes,shadows,arrows}
\usepackage{cite}
\usepackage{placeins}
\usepackage{hyperref}
\usepackage{tcolorbox}
\calclayout                   
\newtheorem{dummy}{anything}[section]
\newtheorem{theorem}[dummy]{Theorem}

\newtheorem{lemma}[dummy]{Lemma}

\theoremstyle{definition}%%Change Theoremstyle
\newtheorem{definition}[dummy]{Definition}

\newtheorem{example}[dummy]{Example}

\newtheorem{remark}[dummy]{Remark}

\newtheorem{question}[dummy]{Question}
\def\:{\mkern 1.2mu \colon}
\newcommand{\mmatrix}[4]{\left (\vcenter
	{\xymatrix@C-2pc@R-2pc{#1&#2\\#3&#4} } \right )}

\DeclareMathOperator{\Mod}{Mod}
\DeclareMathOperator{\id}{id}   
\numberwithin{equation}{section}
\begin{document}
	\title[Minimal Generation of Mapping Class Groups]
	{Minimal Generation of Mapping Class Groups: A Survey of the Orientable Case }
    
	
	\author{ T{\"{u}}l{\.I}n Altun{\"{o}}z, Mehmetc{\.I}k Pamuk, Oguz Yildiz}
		
	\address{Faculty of Engineering, Ba\c{s}kent University, Ankara, Turkey} 
\email{tulinaltunoz@baskent.edu.tr} 
\address{Department of Mathematics, Middle East Technical University,
 Ankara, Turkey}
 \email{mpamuk@metu.edu.tr}
 \address{Department of Mathematics, Middle East Technical University,
 Ankara, Turkey}
  \email{oguzyildiz16@gmail.com}
		
	\date{July 2, 2025}	
	%%%%%%%%%%%%%%%%%%%%%%%%%%%%%%%%%%%%%%%%%%%%%%%%%%%%%%%%%%%%%%%%%%

\begin{abstract}
The mapping class group of an orientable surface, which records its symmetries up to isotopy, plays a central role in low-dimensional topology.
This chapter explores the foundational problem of determining minimal generating sets for these groups.
We chart the development of this area from classical results involving Dehn twist generators to more recent breakthroughs showing that mapping class groups can be generated by just two elements, pairs of torsion elements, or a small collection of involutions.
This chapter contains a discussion of the most current results for punctured surfaces, including a new improvement showing that for an even number of punctures $p\geq 8$ the group $\Mod(\Sigma_{13,p})$ is generated by three involutions.
Throughout, we highlight the rich interplay between the algebraic features of these generating sets and the underlying geometric structures they encode.
The chapter aims to provide a comprehensive account of the pursuit of algebraic and geometric efficiency within one of topology’s most intricate and influential groups.
\end{abstract}

	%%%%%%%%%%%%%%%%%%%%%%%%%%%%%%%%%%%%%%%%%%%%%%%%%%%%%%%%%%%%%%%%%%%%%%%%%%%%%%%
	
	\maketitle	

\section{Introduction}
The mapping class group\index{mapping class group} of an orientable surface is a fundamental object in low-dimensional topology, geometric group theory and algebraic geometry, and is closely related with hyperbolic geometry, $3$- and $4$-manifolds, and symplectic geometry.
It encodes the symmetries of the surface up to isotopy and plays a central role in the study of moduli spaces, Teichmüller theory, and $3$-manifolds.
Understanding how the mapping class group can be generated by various types of elements, such as Dehn twists, involutions, or finite-order elements, offers deep insight into both its algebraic and geometric structure.
Different generating sets reflect different aspects of the group's interaction with geometry, dynamics, and arithmetic.
For instance, Dehn twists\index{Dehn twist} correspond naturally to geometric operations on curves;
torsion elements\index{torsion element} often arise from symmetries or automorphisms of the surface;
and involutions\index{involution} relate closely to Coxeter-type presentations and reflection symmetries.
Finding generating sets\index{generating set} that satisfy particular properties (e.g., minimal size, composed entirely of involutions, elements with minimal orders, elements with the same order or consisting of bounded-order elements) is important for several reasons:

\begin{itemize}
  \item[(i)] \text{Theoretical significance}: Different generating sets can reveal hidden symmetries, simplify presentations, or connect mapping class groups to other well-studied algebraic objects such as braid groups, Coxeter groups, and arithmetic groups.
  \item[(ii)] \text{Computational applications}: Generators of small finite order can simplify algorithms in computational topology and geometry, particularly for computations involving mapping classes, moduli spaces, or $3$-manifold invariants.
  \item[(iii)] \text{Geometric insight}: Some generating sets correspond naturally to geometric decompositions of surfaces (e.g., pants decompositions), leading to a deeper understanding of surface geometry and the structure of moduli spaces.
  \item[(iv)] \text{Connections to dynamics and physics}: In contexts such as string theory, topological quantum field theory and conformal field theory, mapping class groups arise as symmetry groups of moduli spaces.
Specific generating sets often correspond to symmetries of physical systems or algebraic structures in these theories.
\end{itemize}

This chapter provides an overview of the generation of  mapping class groups, focusing in particular on finite generating sets and the structural properties of their elements.
We emphasize both the theoretical richness and the practical importance of constructing generating sets that satisfy constraints such as minimality or bounded order.
\section{Orientable Surfaces}

\subsection{Background and Classical Generators}
Let $\Sigma_{g, p}^b$ denote a connected orientable surface\index{surface!orientable} of genus $g$ with $p \geq 0$ punctures and $b \geq 0$ boundary components.
When $p=0$ or $b=0$, we omit the corresponding index from the notation.
In this section, we briefly define the mapping class group of a surface and introduce its foundational concepts.
Our purpose is to create minimal generating sets for their mapping class groups.
Elements of the mapping class group can be studied via their actions on homotopy classes of simple closed curves.
We begin by defining basic types of curves and then outline the classical classification of surfaces (see \cite{farb-margalit} for details).
\begin{definition}
A \emph{simple closed curve}\index{simple closed curve} on $\Sigma_g$ is the image of an injective continuous map $S^1 \to \Sigma_g$ from the unit circle.
\end{definition}

\begin{definition}
A simple closed curve $c$ is \emph{nonseparating}\index{simple closed curve!nonseparating} if $\Sigma_g \setminus c$ is connected;
otherwise it is \emph{separating}\index{simple closed curve!separating}.
\end{definition}

\begin{definition}
For homotopy classes $[c_1], [c_2]$ of simple closed curves, their \emph{geometric intersection number}\index{geometric intersection number} $i(c_1, c_2)$ is the minimal number of transverse intersection points among representative curves.
\end{definition}

\begin{theorem}[Classification of Surfaces]
There exists a bijection:
\[
\{g \in \mathbb{Z} : g \geq 0\} \longleftrightarrow \{\text{homeomorphism classes of closed, connected, orientable surfaces}\}
\]
\end{theorem}

\begin{theorem}[Curve Mapping Criterion]
There exists an orientation-preserving homeomorphism mapping a curve $c_1$ to $c_2$ if and only if $\Sigma_g \setminus c_1 \cong \Sigma_g \setminus c_2$.
In particular, any two nonseparating curves are related by a homeomorphism.
\end{theorem}

With the basic concepts of surfaces and curves established, we can now formally define the central object of this chapter.
\begin{definition} \label{Mapping Class Group}
Let $\Sigma = \Sigma_{g,p}^{b}$ denote a connected, orientable surface of genus $g$ with $p$ punctures and $b$ boundary components.
The \textbf{mapping class group}\index{mapping class group} of $\Sigma$, denoted by $\Mod(\Sigma)$, is the group of isotopy classes of orientation-preserving homeomorphisms of $\Sigma$ that fix the boundary components pointwise and preserve the set of punctures.
Formally, it is defined as the quotient group:
$$
\Mod(\Sigma_{g,p}^{b}) = \pi_0(\text{Homeo}^+(\Sigma_{g,p}^{b})) := \text{Homeo}^+(\Sigma_{g,p}^{b}) / \text{Homeo}_0(\Sigma_{g,p}^{b})
$$
where $\text{Homeo}^+(\Sigma_{g,p}^{b})$ is the topological group of orientation-preserving homeomorphisms of $\Sigma$ with the given restrictions, and $\text{Homeo}_0(\Sigma_{g,p}^{b})$ is the normal subgroup of homeomorphisms that are isotopic to the identity.
\end{definition}

\subsection{Dehn Twists: Fundamental Generators}
The primary generators of mapping class groups are \emph{Dehn twists}\index{Dehn twist} about simple closed curves.
\begin{definition}
Consider the annulus $I \times S^1$ with coordinates $(x, e^{i\theta})$.
The \emph{twist map} is defined as
\[
t : I \times S^1 \to I \times S^1, \quad (x, e^{i\theta}) \mapsto (x, e^{i(\theta + 2\pi x)}).
\]
\end{definition}

\begin{definition}
Let $a \subset \Sigma_g$ be a simple closed curve with an annular neighborhood $N \cong I \times S^1$.
The \emph{right Dehn twist}\index{Dehn twist!right} about $a$ is the homeomorphism $t_a \colon \Sigma_g \to \Sigma_g$ given by
\[
t_a(x) = 
\begin{cases} 
h^{-1} \circ t \circ h(x) & x \in N \\
x & x \in \Sigma_g \setminus N,
\end{cases}
\]
where $h: N \to I \times S^1$ is an orientation-preserving homeomorphism.
\end{definition}

\subsection{Algebraic Relations}
Dehn twists satisfy fundamental relations\index{Dehn twist!relations} reflecting surface topology.
\begin{itemize}
    \item \textbf{Isotopy}.
$t_a = t_b$ if and only if $a$ is isotopic to $b$.
    \item \textbf{Conjugation}.
$f t_a f^{-1} = t_{f(a)}$ for homeomorphisms $f$.
    \item \textbf{Commutativity}. If $i(a,b) = 0$, then $t_a t_b = t_b t_a$.
    \item \textbf{Braid Relation}\index{braid relation}. If $i(a,b) = 1$, then $t_a t_b t_a = t_b t_a t_b$.
\end{itemize}
    
Other key relations include the \textbf{Lantern Relation}\index{lantern relation}, which applies to curves on a sphere with four holes ($\Sigma_0^4$) as shown in Figure~\ref{lantern}, and is given by $t_a t_b t_c t_d = t_x t_y t_z$.
\begin{figure}[h]
\begin{center}
\scalebox{0.19}{\includegraphics{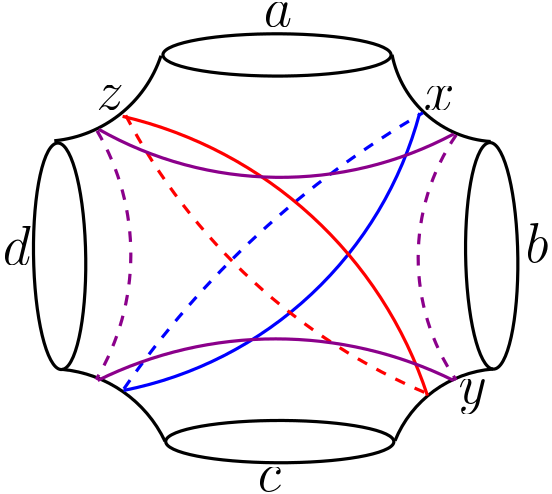}}
\caption{The curves in the lantern relation $t_a t_b t_c t_d = t_x t_y t_z$.}
\label{lantern}
\end{center}
\end{figure}

For surfaces of genus one, we have the \textbf{Two-holed Torus Relation}\index{two-holed torus relation} for $\Sigma_1^2$, illustrated in Figure~\ref{chain_torus}, which is $(t_a t_b t_c)^4 = t_d t_e$;
and the \textbf{One-holed Torus Relation}\index{one-holed torus relation} for $\Sigma_1^1$, depicted in Figure~\ref{chain_torus_one}, given by $(t_a t_b)^6 = t_d$.
\begin{figure}[h]
\begin{center}
\scalebox{0.2}{\includegraphics{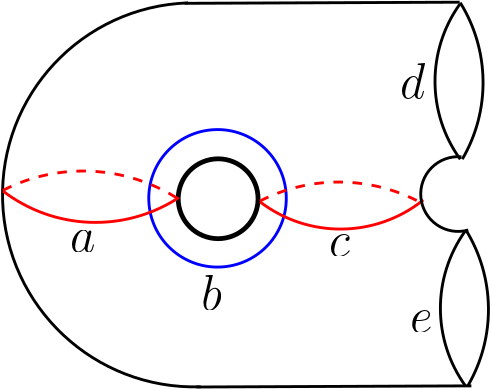}}
\caption{The curves in the two-holed torus relation  $(t_a t_b t_c)^4 = t_d t_e$.}
\label{chain_torus}
\end{center}
\end{figure}

\begin{figure}[h]
\begin{center}
\scalebox{0.3}{\includegraphics{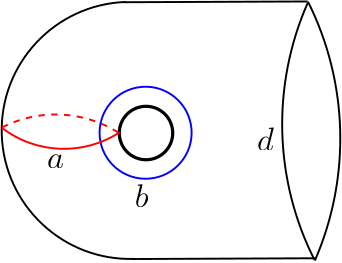}}
\caption{The curves in the one-holed torus relation $(t_a t_b)^6 = t_d$.}
\label{chain_torus_one}
\end{center}
\end{figure}

The mapping class group is a central object in low-dimensional topology, and its algebraic structure is of deep interest.
The problem of finding generators for $\Mod(\Sigma_g)$ was first addressed by Dehn~\cite{dehn}.
He showed that $\Mod(\Sigma_g)$ is generated by finitely many Dehn twists, specifically giving a generating set of $2g(g - 1)$ Dehn twists for $g \geq 3$.
A significant improvement came from Lickorish~\cite{lickorish}, who proved in 1964 that $3g-1$ Dehn twists about nonseparating curves are sufficient to generate the group for any $g \ge 1$.
Later, Humphries~\cite{humphries} refined this result, showing that $2g+1$ Dehn twists suffice.
Humphries also proved that this is the minimal number of Dehn twist generators required for $\text{Mod}(\Sigma_g)$ when $g \ge 1$.
A standard depiction of this generating set is shown in Figure~\ref{humphries}.
\begin{figure}[h]
\begin{center}
\scalebox{0.34}{\includegraphics{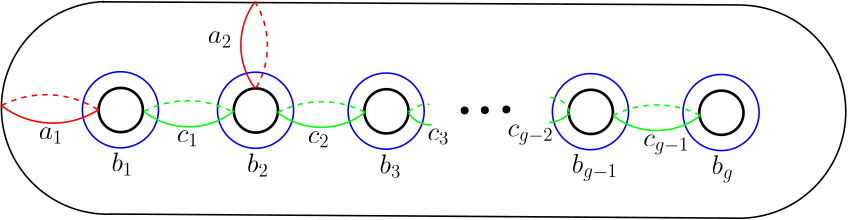}}
\caption{Humphries' generating set\index{Humphries generating set}}
\label{humphries}
\end{center}
\end{figure}

\subsection{Elementary Mapping Class Groups}
\begin{itemize}
    \item [(i)] $\Mod{}(\text{disk}) = \{\text{id}\}$ (Alexander Lemma\index{Alexander Lemma}).
    \item [(ii)]  $\Mod{}(\text{sphere}) = \{\text{id}\}$.
    \item [(iii)] $\Mod{}(\text{annulus}) \cong \mathbb{Z}$, generated by a Dehn twist $t_c$ about the core curve $c$.
    \item [(iv)] $\Mod{}(\text{disk with $n$ marked points}) \cong B_n$ (braid group\index{braid group}).
\end{itemize}

In the study of surfaces and their symmetries, two other fundamental objects emerge: the Teichmüller space and the moduli space.
These interconnected concepts play crucial roles in geometric group theory, particularly in the investigation of generating sets for mapping class groups.
Before we finish this section, let us give their definitions, relationships, and importance in the context of geometric group theory.
Let $\Sigma$ denote an orientable  surface that is either  closed, has punctures, or has boundary components.
\begin{definition}
The \textbf{Teichmüller space}\index{Teichmüller space} $\mathcal{T}(\Sigma)$ is the space of marked complex structures on $\Sigma$ up to isotopy.
Formally:
\[
\mathcal{T}(\Sigma) = \{(X,f) \mid X \text{ Riemann surface}, f: \Sigma \to X \text{ diffeomorphism}\} / \sim
\]
where $(X,f) \sim (Y,g)$ if $g \circ f^{-1}: X \to Y$ is isotopic to a biholomorphism.
\end{definition}

\begin{definition}
The \textbf{moduli space}\index{moduli space} $\mathcal{M}$ is the space of isomorphism classes of complex structures on $\Sigma$:
\[
\mathcal{M} = \{\text{Riemann surfaces of genus } g \text{ with } n \text{ marked points}\} / \text{biholomorphism}.
\]
\end{definition}

These spaces are related through fundamental group actions and quotient constructions:

\begin{enumerate}
\item The mapping class group $\text{Mod}(\Sigma)$ acts properly discontinuously on $\mathcal{T}(\Sigma)$ by changing markings: 
\[
\varphi \cdot [(X,f)] = [(X, f \circ \varphi^{-1})].
\]

\item The moduli space is the quotient of Teichmüller space by this action:
\[
\mathcal{M} \cong \mathcal{T}(\Sigma) / \text{Mod}(\Sigma).
\]

\item For surfaces with $3g-3+n > 0$, $\mathcal{T}(\Sigma)$ is contractible and $\text{Mod}(\Sigma)$ acts with finite stabilizers, making $\mathcal{M}$ an orbifold with fundamental group $\text{Mod}(\Sigma)$.
\end{enumerate}

These objects provide critical frameworks for studying mapping class groups as finitely generated groups:

\begin{itemize}
\item[-] \text{Geometric realization:} The action of $\text{Mod}(\Sigma)$ on $\mathcal{T}(\Sigma)$ endows the group with a geometric structure.
When equipped with the Weil--Petersson metric\index{Weil-Petersson metric}, Teichmüller space becomes a CAT(0) space, allowing the application of geometric group theory techniques.
\item[-]  \text{Group cohomology:} The contractibility of $\mathcal{T}(\Sigma)$ makes it an Eilenberg--MacLane space for $\text{Mod}(\Sigma)$, enabling computation of group cohomology via the orbifold cohomology of $\mathcal{M}$.
\item[-]  \text{Generating set investigations:} The geometric action provides tools to study generating sets:
\begin{itemize}
\item The Nielsen--Thurston classification\index{Nielsen-Thurston classification} of mapping classes (periodic, reducible, pseudo-Anosov) is defined via their action on $\mathcal{T}(\Sigma)$.
\item Geometric intersection numbers on surfaces translate to displacement distances in $\mathcal{T}(\Sigma)$.
\item The Teichmüller space $\mathcal{T}(\Sigma)$, specifically when equipped with the Weil--Petersson metric, possesses strictly negative sectional curvature~\cite{brock}.
This strong geometric property ensures that geodesics diverge exponentially, which is the exact geometric requirement for establishing the necessary separation (disjoint domains) needed to apply the Ping--Pong Lemma and construct free generating sets.
\end{itemize}

\item[-]  \text{Rigidity phenomena:} The action on $\mathcal{T}(\Sigma)$ exhibits various rigidity properties, which constrain possible group actions and inform minimal generation requirements.
\item[-]  \text{Asymptotic geometry:} The large-scale geometry of $\text{Mod}(\Sigma)$, studied via its Cayley graph, is intimately related to the geometry of $\mathcal{T}(\Sigma)$ through the \textit{orbit map}.
This connection is essential for:
\begin{itemize}
\item Understanding growth functions of the group,
\item Studying distortion of subgroups,
\item Analyzing the geometry of word metrics relative to minimal generating sets.
\end{itemize}
\end{itemize}

The geometric perspective informs key results about generating sets for $\text{Mod}(\Sigma)$:
\begin{itemize}
\item[-]  The action on $\mathcal{T}(\Sigma)$ shows that $\text{Mod}(\Sigma)$ is generated by Dehn twists since these act as twist isometries on Teichmüller space
\item[-]  The non-uniform negative curvature of $\mathcal{T}(\Sigma)$ (with the Weil--Petersson metric) explains why certain generating sets (like Humphries' $2g+1$ twists) are minimal - removing generators creates geometric obstructions to connecting orbits
\item[-]  Pseudo-Anosov elements\index{pseudo-Anosov mapping class} (detected by their translation length in $\mathcal{T}(\Sigma)$) serve as test elements when examining independence in generating sets
\item[-]  The quotient $\mathcal{M} = \mathcal{T}(\Sigma)/\text{Mod}(\Sigma)$ has orbifold fundamental group $\text{Mod}(\Sigma)$, making topological properties 
of moduli space correspond to group-theoretic properties of generators
\end{itemize}

Mapping class groups act naturally on both Teichmüller space and the moduli space, and these actions yield concrete geometric insights into $\Mod(\Sigma)$.
For example, Masur and Minsky’s hierarchy machinery in the curve complex\index{curve complex} provides explicit diameter and distance estimates, which in turn give bounds on the minimal number of Dehn twists needed to generate $\Mod(\Sigma)$ \cite{masur-minsky}.
Similarly, Brock’s comparison of the Weil–Petersson metric with the pants graph shows how coarse geometry governs word‑length growth \cite{brock}.
Together, these frameworks equip us with a unified set of tools to unravel the combinatorial complexity and algebraic depth of surface mapping class groups.
\begin{remark}
The natural surjection 
$$
\Mod(\Sigma_g)\;\longrightarrow\; \textrm{Sp}(2g,\mathbb Z)
$$
induced by the action on \(H_1(\Sigma_g;\mathbb Z)\), carries each Dehn twist to a symplectic transvection.
Hence any  generating set for $\Mod(\Sigma_g)$ descends to a generating set for $\textrm{Sp}(2g,\mathbb Z)$.
\end{remark}

\section{Minimal Generation of $\Mod(\Sigma_g)$} \label{sec:wajnryb-korkmaz}
Understanding minimal generating sets\index{generating set!minimal} for mapping class groups is a fundamental problem in geometric topology and group theory.
The mapping class group $\Mod(\Sigma_g)$ of a closed surface $\Sigma_g$ is finitely generated, but the structure and properties of its generating sets can vary widely depending on the types of elements allowed (e.g., Dehn twists, torsion elements, involutions).
Studying minimal generating sets provides insight into the algebraic complexity of $\mathrm{Mod}(\Sigma_g)$ and allows for more efficient solutions to algorithmic problems, such as the word and conjugacy problems.
Moreover, minimal generators often admit natural geometric or dynamical interpretations, revealing the essential building blocks of surface diffeomorphisms up to isotopy.
These generating sets also play a central role in understanding actions of $\mathrm{Mod}(\Sigma_g)$ on various combinatorial and geometric complexes, and they connect deeply with applications in Teichmüller theory, $3$-manifold topology, and topological quantum field theory.
Building on decades of work, the following result reduces the generator set  to its absolute minimum, revealing an unexpectedly tight algebraic framework for mapping class groups.
\begin{theorem}[Wajnryb \cite{wajnryb}]\label{thm:wajnryb}
The mapping class group of a compact orientable surface of genus $g\geq 1$, closed or with one boundary component, can be generated by two elements.
\end{theorem}

This result is surprising because the mapping class group is highly non-abelian and complex, and previous generating sets were much larger.
Wajnryb's theorem reveals remarkable economy in mapping class group structure.
The proof demonstrates deep interplay between geometric topology and combinatorial group theory, influencing subsequent work on:
\begin{itemize}
\item[(i)] Finite generation questions for related groups,
\item[(ii)] geometric realization problems,
\item[(iii)] complexity of mapping class group algorithms.
\end{itemize}

In $2005$, Korkmaz proved that one of these generators can be taken as a Dehn twist.
\begin{theorem}[Korkmaz \cite{korkmaz2005}]\label{thm:korkmaz}
For any closed oriented surface \(\Sigma_g\) of genus \(g \geq 3\), the mapping class group $\mathrm{Mod}(\Sigma_g)$ is generated by (see Figure~\ref{humphries}):
\begin{itemize}
\item[(i)] A single Dehn twist $t_{a_2}$ about a non-separating curve $a_2$,
\item[(ii)] A torsion element $S=t_{b_g} t_{c_{g-1}} t_{b_{g-1}} \ldots t_{c_1} t_{b_1} t_{a_1}$ of order $4g+2$ acting as a rotational symmetry.
\end{itemize}
\end{theorem}

 Korkmaz shows that the subgroup $G = \langle t_{a_2}, S \rangle$ contains Dehn twists about all Humphries curves.
This is accomplished by analyzing the orbit of the curve $a_2$ under conjugation by powers of $S$.
Once one $a_i$ is recovered, the cyclic action of $S$ generates all other $a_i$, and hence the full Humphries generating set.
Therefore, $G$ contains all Humphries twists, and hence $\Mod(\Sigma_g)$.

 The significance of these minimal generating sets extends beyond their algebraic compactness, offering a new geometric framework for understanding the action of the mapping class group.
Korkmaz’s generating pair is particularly insightful in this regard. The Dehn twist provides a localized geometric action tied to a specific curve, while the rotational symmetry $S$ acts globally, organizing the topology of the entire surface.
The combined action of these two generators on Teichmüller space, $\mathcal{T}(\Sigma_{g})$, is rich enough to describe the entire group's dynamics.
This interplay between a local twist and a global symmetry allows for a more direct analysis of the geometry of $\mathcal{T}(\Sigma_{g})$ and the structure of its resulting quotient orbifold, the moduli space.
Furthermore, the very minimality of this generating set simplifies the study of complex dynamical phenomena;
key properties, such as the ergodicity of the mapping class group action or the distribution of pseudo-Anosov axes, can now be investigated by analyzing compositions of just these two specific geometric transformations \cite{farb-margalit}.
\text{Impact on Teichmüller Theory:}  
  
\begin{itemize}  
    \item [(i)]These generators give explicit coordinates for studying the quotient orbifold $\mathcal{T}(\Sigma_g)/\mathrm{Mod}(\Sigma_g)$, which is the moduli space of Riemann surfaces.
    \item [(ii)] They provide a concrete way to study pseudo-Anosov elements through products of these generators.
\end{itemize}  

\noindent The geometric insights provided by these minimal generators extend naturally to the realm of $3$-manifolds through surface dynamics.
The Dehn twist-torsion pair $(t_{a_2}, \rho)$ not only coordinates the action on Teichmüller space but also directly governs the construction of $3$-manifolds via Heegaard splittings\index{Heegaard splitting}~\cite{birman2001}.
Specifically, the rotational symmetry $\rho$ encodes periodic identifications of handlebodies, while compositions of $t_{a_2}$ and $\rho$ produce the essential complexity required to build hyperbolic structures~\cite{thurston1982}.
This duality reflects a deep reciprocity: the generators' efficiency in describing Teichmüller geodesics translates to algorithmic advantages in $3$-manifold topology, where minimal generating sets simplify computations of Heegaard Floer homology\index{Heegaard Floer homology} and surgery coefficients~\cite{ozsvath-szabo2004}.
Furthermore, the relationship between pseudo-Anosov elements (arising from products of $t_{a_2}$ and $\rho$) and hyperbolic $3$-manifold geometries becomes computationally tractable precisely because the generators' limited number constrains the combinatorial complexity of gluing maps.
Thus, the geometric economy achieved in Teichmüller theory propagates to practical advances in $3$-manifold classification and invariant calculation.
\text{Applications to $3$-Manifold Topology}:  
  
\begin{itemize}  
    \item [(i)] Every closed oriented $3$-manifold can be obtained by performing Dehn surgery on a link in $S^3$, or alternatively by gluing two handlebodies via a surface homeomorphism~\cite{birman2001}.
    \item [(ii)]These generators provide a minimal way to describe such gluing maps, leading to:  
        \begin{itemize}  
            \item [-] Simplified presentations of $3$-manifolds through Heegaard splittings.
            \item [-]  More efficient algorithms for computing $3$-manifold invariants.
            \item [-]  Better understanding of the relationship between surface maps and $3$-manifold structures.
        \end{itemize}  
    \item[(iii)] The generators give a way to study how properties of $3$-manifolds (like hyperbolic structures) depend on the gluing map~\cite{thurston1982}.
    \item [(iv)] They provide insight into the Heegaard Floer homology through their action on curves and handlebodies~\cite{ozsvath-szabo2004}.
\end{itemize}  
\section{GENERATION BY TORSION ELEMENTS}
A major development in the algebraic understanding of mapping class groups is the realization that they can be generated by elements of finite order\index{torsion element!generators}.
This connects the theory to finite group actions, modular forms, and the automorphism groups of Riemann surfaces.
Particularly notable is the discovery that small-order torsion elements often suffice to generate the entire group.
The study of generating sets composed entirely of torsion elements dates back to Maclachlan~\cite{maclachlan}.
He proved that $\Mod(\Sigma_g)$ is generated by conjugates of just two elements, of orders $2g + 2$ and $4g + 2$, and applied this to show the simple connectivity of the moduli space of Riemann surfaces.
Since $\operatorname{Mod}(\Sigma_g)$ is non-cyclic for $g \geq 1$, it cannot be generated by a single element, making the minimal number of torsion generators a fundamental question.
Subsequent work has progressively refined these bounds:
\begin{itemize}
    \item[(i)] Lu~\cite{Lu} constructed a generating set for $\Mod(\Sigma_g)$ consisting of three elements, two of which are of finite order.
These are: a Dehn twist $t_{a_1}$ about the simple closed curve $a_1$;
a composition $N$ of five Dehn twists along five carefully chosen curves supported near the first two handles of the surface;
and a rotational symmetry $R$ of order $g$ that cyclically permutes the handles.
Lu showed that all of Lickorish’s generators can be expressed in terms of these three elements.
Specifically, the curves $a_i$ (see Figure~\ref{humphries}) are obtained as conjugates of $a_1$ under powers of $R$, while the $b_i$ curves arise as conjugates of $N^{-1}t_{a_1}N$ under powers of $R$.
    \item[(ii)]Wajnryb~\cite{wajnryb} constructed a generating set for $\Mod(\Sigma_g)$ consisting of just two elements, one of which is of finite order.
He defined the product of Dehn twists along standard curves (see Figure~\ref{humphries}) $S=t_{a_1}t_{b_1}t_{c_1}t_{b_2}t_{c_2} \ldots t_{b_{g-1}}t_{c_{g-1}}t_{b_g}$ and $R= t_{a_{g-1}}t^{-1}_{a_{g}}$, where the $a_i$'s are as in Figure~\ref{NG}.
He showed that all of Humphries’ generators can be expressed as combinations of these two elements, thereby establishing an efficient generating set for the mapping class group.
    \item[(iii)] Korkmaz (recall Theorem~\ref{thm:korkmaz}) improved Wajnryb's result by constructing a generating set consisting of two elements: the finite-order element $S = t_{b_g} t_{c_{g-1}} \cdots t_{a_1}$ and the Dehn twist $t_{a_2}$.
    \item[(iv)]  Monden~\cite{Monden} showed generation by three elements of order~$3$, or four of order~$4$.
His approach involves cutting the surface along the standard Humphries generators to decompose it into $g/2$ or $(g-1)/2$ surfaces depending on whether the genus $g$ is even or odd.
On each of these subsurfaces, he defines rotational elements of order~$3$, and introduces two $2$-chains such that their composition with the rotational element also yields an element of order~$3$.
Using these three torsion elements, he was able to recover all of the Humphries generators.
A similar strategy was employed to construct four generators of order~$4$.
    \item[(v)] Lanier~\cite{lanier2018} proved that for any integer $k \geq 6$, the mapping class group $\Mod(\Sigma_g)$ is generated by three elements of order $k$ provided that the genus $g$ satisfies $g \geq 2((k-1)^2 + 1)$.
His construction begins with a $k$-fold symmetric embedding $S_g$ of $\Sigma_g$ in $\mathbb{R}^3$ where a rotation $r$ of $\Sigma_g$  by an angle of $2\pi/k$ about its axis yields an element of order $k$ similar to~\cite{Monden}.
To obtain the remaining elements of order $k$, Lanier embeds $k$-symmetric subsurfaces $\sigma_i$ into $\Sigma_g$	 and within each $\sigma_i$, he places a chain of simple closed curves of length $2g_i-1$, where $g_i$ is the genus of $\sigma_i$.
He then constructs the required elements by mapping $S_g$ to $\Sigma_g$, applying the rotation $r$, and mapping back.
This careful use of symmetry and surface decomposition yields three torsion elements of order $k$ that generate the entire group.
    \item[(vi)]  Yildiz~\cite{yildiz2} achieved generation by two elements following similar ideas as in   Theorem \ref{thm:korkmaz-involution}.
He obtained
    \begin{itemize}
        \item For $g \geq 6$, two elements of order $g$;
        \item For $g \geq 7$, two elements of orders $g$ and $g'$ (where $g' > 2$ is the smallest divisor of $g$);
        \item For $g \geq 3k^2 + 4k + 1$ and $k \in \mathbb{Z}^+$, two elements of order $g / \gcd(g,k)$.
    \end{itemize}
    
\end{itemize}

\noindent The pursuit of torsion generators is motivated by their intrinsic advantages:

\begin{itemize}
  \item[(i)] Simplicity: Low-order elements (e.g., involutions or order-$3$ torsion) admit concise algebraic descriptions and facilitate computational manipulation.
  \item[(ii)] Symmetry: Torsion elements often arise from geometric symmetries (e.g., hyperelliptic involutions or rotational symmetries), linking group structure to surface geometry.
  \item[(iii)] Representation Theory: They simplify the construction of linear and modular representations of $\Mod(\Sigma_g)$.
  \item[(iv)] Topological Interpretability: Relations between torsion generators frequently encode geometric operations (e.g., rotations or reflections), offering intuitive decompositions of mapping classes.
\end{itemize}

These constructions reveal the deep interplay between geometry and group theory.
That such complex groups can be generated by a few symmetric, low-order elements highlights their rigidity and hidden structure.
\subsection{Involution Generators}\label{subsec:involution}
Involutions\index{involution} often correspond to geometric symmetries of the surface, such as reflections.
Generating the mapping class group by such elements\index{mapping class group!involution generators} can provide a more intuitive and geometrically meaningful way to decompose complex surface transformations into simpler, more symmetric building blocks.
This approach can illuminate the relationship between the algebraic structure of $\Mod(\Sigma_g)$ and the geometric symmetries of $\Sigma_g$, potentially leading to new insights and computational advantages.
The fact that the mapping class group for a closed surface of genus $g \ge 3$ can be generated by involutions was first established by McCarthy and Papadopoulos~\cite{mccarthy-papadopoulos}.
Furthermore, the problem of generating a group by elements of a specific type is a common theme in group theory.
For the mapping class group, the focus on involutions is motivated by several factors.
Firstly, involutions are often easier to understand and can be classified geometrically.
For instance, on a closed orientable surface, involutions are related to hyperelliptic involutions\index{hyperelliptic involution} and other symmetries.
Secondly, the existence of a generating set consisting of involutions can have implications for the group's algebraic properties, such as its residual finiteness or its relationship to other well-studied groups.
Finally, exploring generation by involutions can lead to connections with other areas of mathematics, such as the theory of braid groups and Coxeter groups, which also have natural representations involving reflections or elements of order two.
Thus, the investigation into generating the mapping class group by involutions is driven by a desire for geometric clarity, algebraic insight, and potential connections to broader mathematical frameworks.
Building upon these motivations, significant progress has been made in determining the minimal number of involutions required to generate the mapping class group.
A notable result in this direction is the following:

\begin{theorem} \cite{brendle-farb} \label{thm:brendle-farb}
For $g \geq 3$, $\operatorname{Mod}(\Sigma_g)$ is generated by $6$ involutions.
\end{theorem}

The proof that the mapping class group $\operatorname{Mod}(\Sigma_g)$is generated by $6$ involutions is constructive.
The strategy is to start with a known, small set of generators and show that each can be expressed as a product of elements from a set of just six involutions.
The authors begin with a known generating set for $\operatorname{Mod}(\Sigma_g)$ consisting of four elements: three Dehn twists, $t_{\alpha}$, $t_{\beta}$, $t_{\gamma}$, and a rotation $R_g$.
The rotation $R_g$ together with the three twists generate all of Lickorish's $3g-1$ twist generators.
The core of the proof is to express these four generators using a small, common set of involutions.
\begin{enumerate}
    \item[(i)] \text{The Rotation $R_g$:} The rotation $R_g$ is shown to be the product of two involutions, $\rho_1$ and $\rho_2$, which are rotations by $\pi$ about specific axes.
So, $R_g = \rho_1 \rho_2$. These are the first two of the six required involutions.
    \item[(ii)] \text{The Lantern Relation:} To decompose the Dehn twists, the authors use the lantern relation.
A key insight is showing that a Dehn twist about a non-separating curve can be written as the product of four involutions.
    \item[(iii)] \text{The Good Lantern Embedding:} A crucial step is the specific embedding of a good lantern into the surface such that the curves $\alpha$, $\beta$, and $\gamma$ from the generating set serve as three of its four boundary components (denoted $a_1$, $a_3$, and $a_4$ respectively) as shown in Figure~\ref{goodlantern}.
    \begin{figure}[h]
\begin{center}
\scalebox{0.25}{\includegraphics{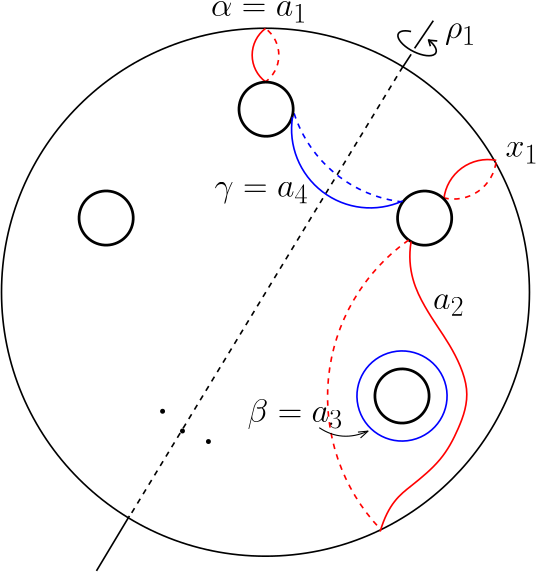}}
\caption{A good lantern embedded with three boundary curves $\alpha$, $\beta$ and $\gamma$ and one interior curve $x_1$.}
\label{goodlantern}
\end{center}
\end{figure}

    \item[(iv)] \text{Expressing Dehn Twists:} With this setup, the authors use a series of pair-swap involutions ($J_1, J_2, J_3$) in combination with the existing involution $\rho_1$ to build the required Dehn twists.
    \begin{itemize}
        \item[-] $t_\gamma$ is generated using $\rho_1$ and three new involutions: $X_1\rho_1X_1^{-1}$, $J_1$, and $J_2$.
However, since $X_1$ is a twist about the curve $\alpha$, and $\rho_1$ takes $\alpha$ to $x_1$, the term $X_1\rho_1X_1^{-1}$ is conjugate to $\rho_1$ itself.
The truly new involutions are $J_1$ and $J_2$.
        \item[-] $t_\beta$ is then generated using the same set of involutions plus one new pair-swap involution, $J_3$.
        \item[-] $t_\alpha$ is shown to be generated using a fourth pair-swap, $J_4$, bringing the total count to $2+3+1+1=7$ involutions ($\rho_1, \rho_2, J_1, J_2, J_3, J_4$ and a conjugate of $\rho_1$).
    \end{itemize}
\end{enumerate}

The final step in the proof, attributed to  Kassabov, is the observation that the involution $J_4$ is redundant.
It can be expressed as a product of other involutions already in the set: $J_4 = J_2J_3J_2$.
This eliminates the need for a seventh involution, completing the proof that $\Mod(\Sigma_g)$ can be generated by just $6$ involutions.
\begin{remark}
For closed orientable surfaces of genus \(g \leq 2\), the mapping class group \(\operatorname{Mod}(\Sigma_g)\) cannot be generated entirely by involutions.
This limitation arises from structural and algebraic properties specific to low-genus cases:

For $g=1$, the mapping class group satisfies:
\[
\operatorname{Mod}(\Sigma_1) \cong \operatorname{SL}(2, \mathbb{Z}).
\]
Its abelianization is:
\[
H_1(\operatorname{Mod}(\Sigma_1); \mathbb{Z}) \cong \mathbb{Z}/12\mathbb{Z}.
\]
Since involutions map to torsion elements of order $\leq 2$ in the abelianization, they cannot generate the full group, which contains elements of order $12$.
For $g=2$, the mapping class group \(\operatorname{Mod}(\Sigma_2)\) has abelianization:
\[
H_1(\operatorname{Mod}(\Sigma_2); \mathbb{Z}) \cong \mathbb{Z}/10\mathbb{Z}.
\]
Again, involutions cannot generate elements of order $10$ in the abelianization.  This is insufficient to span the full group.
\end{remark}

Kassabov~\cite{kassabov} substantially refined the above result by proving that for sufficiently large genus, 
only four involutions are required, a reduction highlighting deeper symmetries in high-genus surfaces.
This improvement not only approached conjectured minimal bounds for involution generators but also underscored the role of geometric patterns in simplifying generating sets.
Kassabov further established connections between minimal generation, torsion elements, and surface symmetries, 
laying groundwork for subsequent work on bounded generation across topological and arithmetic groups.
\begin{theorem} \cite{kassabov}  \label{thm:kassabov}
Let $\Sigma_g$ be a closed, oriented surface of genus $g\ge7$.  There exist elements
$$\rho_1,\rho_2,\rho_3,\rho_4\in\Mod(\Sigma_g)$$
of order two such that
$$\Mod(\Sigma_g)=\langle\rho_1,\rho_2,\rho_3,\rho_4\rangle.
$$
\end{theorem}

\begin{proof}
We sketch the main ideas of Kassabov’s construction:
Start with a model of $\Sigma_g$ as a symmetric embedding in $\mathbb{R}^3$, so that two reflections around perpendicular planes give involutions $\rho_1$ and $\rho_2$ acting as half-turn rotations of the surface.
Define two further involutions $\rho_3$ and $\rho_4$ by composing the rotations $\rho_1,\rho_2$ with a correspondence that swaps pairs of handles and reverses orientation locally.
These are arranged so that every Humphries curve lies in the orbit of some curve under the group generated by $\rho_1,\dots,\rho_4$.
\begin{figure}[h]
\begin{center}
\scalebox{0.35}{\includegraphics{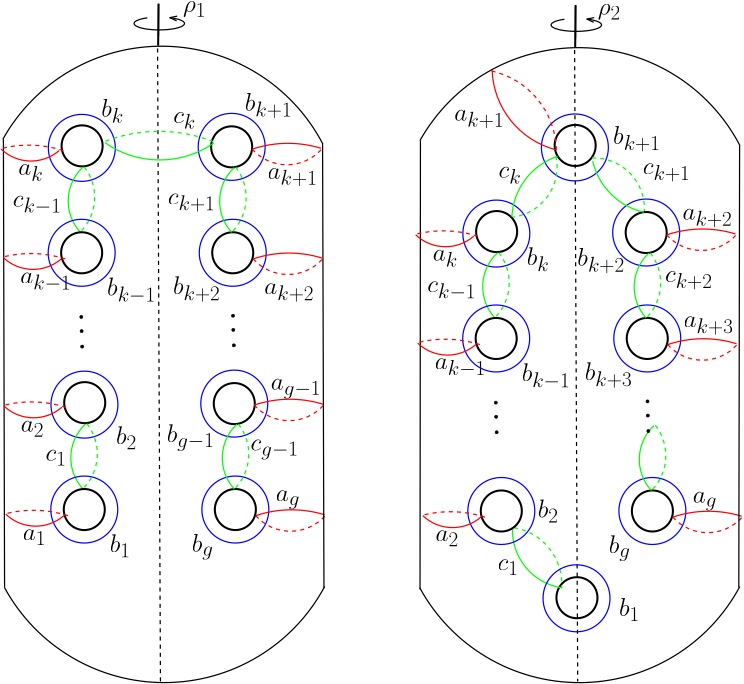}}
\caption{The rotations $\rho_1$ and $\rho_2$ in Kassabov's generating set}
\label{kassabov_fig}
\end{center}
\end{figure}
Observe that the product of two reflections along intersecting axes is a Dehn twist (or its inverse) about the corresponding curve.
By choosing the axes appropriately, one shows that the four involutions generate all Humphries twists.
Since the curve complex $\mathcal{C}(\Sigma_g)$ is connected for $g\ge7$, the subgroup generated by $\{\rho_i\}$ acts transitively on nonseparating curves.
Hence once one twist is generated, all are.
Combining these steps yields the desired four-involution generating set.
\end{proof}

Kassabov's work established that four involutions suffice for generating $\mathrm{Mod}(\Sigma_g)$ when the genus is large enough.
This naturally led to the question of whether this number could be reduced further.
It is known that $\mathrm{Mod}(\Sigma_g)$ cannot be generated by two involutions because it contains non-abelian free groups.
Thus, the theoretical minimum number of involution generators is three.
Korkmaz provided the definitive answer to this question, achieving this minimal bound and proving that three involutions are indeed sufficient for surfaces of sufficiently high genus.
\begin{theorem} \cite{korkmaz2020}\label{thm:korkmaz-involution} 
If $g \geq 8$, the mapping class group $\text{Mod}(\Sigma_g)$ is generated by three involutions.
\end{theorem}

Korkmaz also notes that $\text{Mod}(\Sigma_g)$ is generated by four involutions for $g \geq 3$.
Thus, for $g \geq 8$, three is the minimal number of involutions required to generate $\text{Mod}(\Sigma_g)$.
This result is significant because it establishes the precise minimal number of involution generators for mapping class groups of surfaces with genus $g \geq 8$.
It answers a question posed by Luo \cite{luo}  regarding whether a universal upper bound exists for the number of involution elements needed to generate $\text{Mod}(\Sigma_g)$.
Previous results had shown generation by six involutions\cite{brendle-farb} and later four involutions for $g \geq 7$ \cite{kassabov}.
Brendle and Farb also posed the problem of finding a generating set of size three for $\text{Mod}(\Sigma_g)$ \cite{brendle-farb}.
Korkmaz's paper improves upon these by reducing the number of required involutions to three for a large class of surfaces \cite{korkmaz2020}.
The generation of $\text{Mod}(\Sigma_g)$ also has implications for the symplectic group $\text{Sp}(2g, \mathbb{Z})$, as there is a surjective homomorphism from $\text{Mod}(\Sigma_g)$ onto $\text{Sp}(2g, \mathbb{Z})$.
Thus, a corollary is that $\text{Sp}(2g, \mathbb{Z})$ is also generated by three involutions for $g \geq 8$.
\text{Sketch of the Proof (for $g \ge 8$)}:
The proof is constructive but follows an indirect strategy.
Instead of showing that the three involutions can directly produce a standard set of generators, the approach is to first establish a new, intermediate generating set for the mapping class group and then demonstrate that this new set can be generated by the three chosen involutions.
The argument has two main stages:

\begin{enumerate}
    \item \text{A New Generating Set.} The core of the proof first establishes that for $g \ge 8$, the mapping class group $\Mod(\Sigma_g)$ is generated by the following three elements:
    \begin{itemize}
        \item $\rho_1$, an involution given by rotation of the surface by $\pi$ about the $z$-axis.
        \item $\rho_2$, an involution given by rotation by $\pi$ about a specific line in the $yz$-plane.
        \item $F_1 = t_{b_{1}}t_{a_{2}}t_{c_{3}}t_{c_{4}}^{-1}t_{a_{6}}^{-1}t_{b_{7}}^{-1}$, a specific product of six Dehn twists, where the letters correspond to twists about the curves shown in the paper's figures (see also Figure~\ref{NG}).
    \end{itemize}
     The proof of this result is highly technical, relying on a series of intricate calculations  and the lantern relation to show that this set generates products of opposite Dehn twists (e.g., $t_{a_1}t_{a_2}^{-1}$) and a rotation $R=\rho_1\rho_2$, which are themselves shown to form a generating set.
    \begin{figure}[h]
\begin{center}
\scalebox{0.25}{\includegraphics{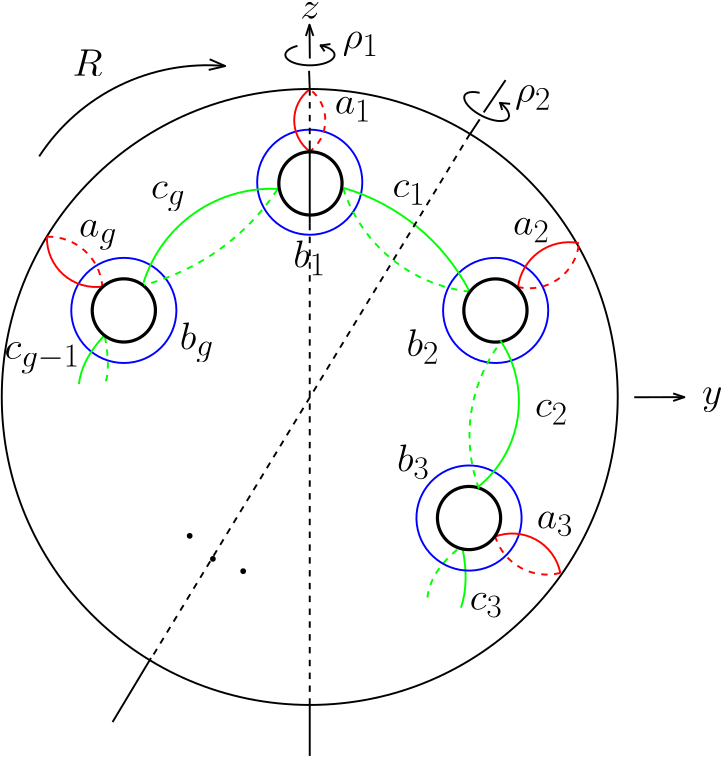}}
\caption{The involutions $\rho_1$ and $\rho_2$  on the surface $\Sigma_g$.}
\label{NG}
\end{center}
\end{figure}

    \item \text{Constructing the Three Involution Generators.} The second stage demonstrates that the generating set $\{\rho_1, \rho_2, F_1\}$ can be replaced by a set of three involutions, $\{I_1 = \rho_1, I_2 = \rho_2, I_3 = \rho_3 F_1\}$ where $\rho_3$ is another geometric involution defined as $\rho_3 = R^2\rho_1R^{-2}$.
It is shown that $I_3$ is indeed an involution by leveraging a key lemma, which states that if $\rho$ is an involution mapping an element $x$ to $y$, then $\rho xy^{-1}$ is also an involution.
In this case, $\rho_3$ conjugates the element $t_{b_1}t_{a_2}t_{c_3}$ to $t_{b_7}t_{a_6}t_{c_4}$, which guarantees that $I_3 = \rho_3 (t_{b_1}t_{a_2}t_{c_3})(t_{b_7}t_{a_6}t_{c_4})^{-1}$ is an involution.
The subgroup generated by $\{I_1, I_2, I_3\}$ is the same as the one generated by $\{\rho_1, \rho_2, F_1\}$.
This is because $I_1$ and $I_2$ are $\rho_1$ and $\rho_2$, and from them one can construct the rotation $R=\rho_1\rho_2$ and subsequently the involution $\rho_3$.
Once $\rho_3$ is obtained, the generator $F_1$ can be recovered from $I_3$ via the product $F_1 = \rho_3I_3$.
Therefore, the three involutions generate $\Mod(\Sigma_g)$ for $g \ge 8$.
\end{enumerate}

Building on Korkmaz’s work, Yildiz \cite{yildiz2020} improved the genus bound to $g \ge 6$ using analogous techniques.
\begin{question}
Can the mapping class group be generated by three involutions for $3 \le g \le 5$?
This remains an open and compelling question.
\end{question}
\section{Surfaces with Punctures}

The introduction of punctures to a surface\index{surface!punctured} significantly alters the structure of its mapping class group.
A key distinction arises between homeomorphisms that permute punctures and those that fix them pointwise.
This leads to the definition of two closely related groups: the (full) mapping class group, which allows puncture permutations, and the pure mapping class group, which does not.
Let $\Sigma_{g,p}$ be a connected orientable surface of genus $g$ with a set $P = \{z_1, \dots, z_p\}$ of $p \geq 1$ marked points, or punctures.
The \textbf{ mapping class group}, denoted $\text{Mod}(\Sigma_{g,p})$, is the group of isotopy classes of homeomorphisms $f: \Sigma_{g,p} \to \Sigma_{g,p}$ that may permute the punctures in $P$.
The \textbf{pure mapping class group}\index{pure mapping class group}\index{mapping class group!pure}, $\text{PMod}(\Sigma_{g, p})$, is the subgroup of $\text{Mod}(\Sigma_{g,p})$ consisting of elements that fix each puncture in $P$ individually.
The relationship between these two groups is described by the following short exact sequence:
$$1 \to \text{PMod}(\Sigma_{g,p}) \hookrightarrow \text{Mod}(\Sigma_{g,p}) \xrightarrow{\pi} S_p \to 1$$
where $S_p$ is the symmetric group on $p$ letters and the surjection $\pi$ records the permutation induced by a mapping class on the set of punctures.
This sequence shows that $\text{PMod}(\Sigma_{g,p})$ is a normal subgroup of $\text{Mod}(\Sigma_{g, p})$ of index $p!$.
\subsection{The Birman Exact Sequence and Point-Pushing Maps} \label{sec:birman}

A powerful tool for understanding generators of $\mathrm{PMod}(\Sigma_{g, p})$ is the Birman exact sequence\index{Birman exact sequence} (also known as the Birman--Hilden exact sequence or forgetful map sequence).
For $p \geq 1$, let $p_i$ be one of the punctures.
There is a forgetful homomorphism $F: \mathrm{PMod}(\Sigma_{g, p}) \to \mathrm{PMod}(\Sigma_{g, p-1})$ obtained by simply forgetting the puncture $p_i$.
This map is surjective.

\begin{theorem}[Birman Exact Sequence]
For $g \geq 0$ and $p \geq 1$, there is a short exact sequence:
$$ 
1 \to \pi_1 (\Sigma_{g, p-1}, p_i) \xrightarrow{\text{push}} \mathrm{PMod}(\Sigma_{g, p}) \xrightarrow{F} \mathrm{PMod}(\Sigma_{g, p-1})  \to 1 
$$
Here, $\Sigma_{g, p-1}$ is the surface $\Sigma_{g, p}$ with the puncture $p_i$ filled in or ignored, and $\pi_1 (\Sigma_{g, p-1}, p_i)$ is its fundamental group based at the location of the forgotten puncture $p_i$.
\end{theorem}

The kernel of $F$ consists of mapping classes that become trivial when $p_i$ is forgotten.
These are precisely the \textbf{point-pushing maps}\index{point-pushing map} (or push maps).
An element $\alpha \in \pi_1 (\Sigma_{g, p-1}, p_i)$, represented by a loop $\alpha$ based at $p_i$, gives rise to a mapping class $\text{push}(\alpha) \in \mathrm{PMod}(\Sigma_{g, p}$ by dragging the puncture $p_i$ around the loop $\alpha$.

Given a simple closed curve $\alpha$ on $\Sigma_{g,p-1}$ passing through the puncture $p_n$, let $A$ be an annular neighborhood of $\alpha$ in $\Sigma_{g, p-1}$. When $p_i$ is considered a puncture within this annulus, the map $\text{push}(\alpha)$ can be expressed as a product of Dehn twists. Specifically, if $c_1$ and $c_2$ are the boundary curves of the annulus 
$A$ (where $p_n$ lies between $c_1$ and $c_2$), then 
$\text{push}(\alpha) = c_1 c_2^{-1}$. These curves $c_1, c_2$ are essential in $\Sigma_{g, p}$ and encircle $p_i$ along with parts of $\alpha$.

The Birman exact sequence allows for an inductive approach to finding generators for $\mathrm{PMod}(\Sigma_{g, p})$:
\begin{enumerate}
    \item Start with known generators for $\mathrm{PMod}(\Sigma_{g, 0}) = \Mod(\Sigma_g)$ (if $i=0$) or $\mathrm{PMod}(\Sigma_{0, p})$ (if $g=0$).
    \item To get generators for $\mathrm{PMod}(\Sigma_{g, p})$ from those of $\mathrm{PMod}(\Sigma_{g, p-1})$, one needs to add lifts of the generators of $\mathrm{PMod}(\Sigma_{g, p-1})$ to $\mathrm{PMod}(\Sigma_{g, p})$ (which can often be chosen to be the same Dehn twists on curves disjoint from $p_n$) plus a set of generators for the kernel $\pi_1(\Sigma_{g, p-1}, p_i)$.
\end{enumerate}
Since $\pi_1(\Sigma_{g, p-1})$ is finitely generated by loops, the corresponding point-pushing maps provide the additional Dehn twist generators needed.
\subsection{Generators for the Full Mapping Class Group}

To generate the full group $\text{Mod}(\Sigma_{g,p})$, one needs elements that can permute the punctures.
Dehn twists alone cannot do this, as they are defined by homeomorphisms supported in an annulus, which can be chosen to be disjoint from the punctures.
The necessary generators are half-twists.

\begin{definition}
Let $\alpha$ be an arc in $\Sigma_{g,p}$ connecting two distinct punctures, say $z_i$ and $z_{j}$.
Let $N$ be a regular neighborhood of $\alpha \cup \{z_i, z_j\}$, which is a disk with two punctures.
The half-twist\index{half-twist} about $\alpha$, denoted $H_{\alpha}$, is the isotopy class of a homeomorphism supported on $N$ that swaps the punctures $z_i$ and $z_j$ by rotating the disk $N$ by an angle of $\pi$.
The square of a half-twist, $H_\alpha^2$, is a Dehn twist about the boundary curve $\partial N$.
\end{definition}

\begin{figure}[h!]
% Placeholder for figure showing a half-twist
\centering
\includegraphics[width=0.55\textwidth]{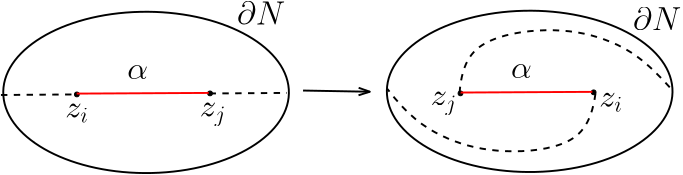}
\caption{A half-twist $H_{\alpha}$ about an arc $\alpha$ connecting the punctures $z_i$ and $z_j$.
The twist swaps the two punctures. Its square, $H_\alpha^2$, is a Dehn twist about the curve $\partial N$.}
\label{fig:half-twist}
\end{figure}

The half-twists generate the symmetric group action on the punctures.
A standard choice is to use the elementary half-twists $H_i$ about arcs connecting consecutive punctures $z_i$ and $z_{i+1}$ for $i = 1, \dots, p-1$.
These generate $S_p$.

\begin{theorem}\cite{farb-margalit}\label{thm:punctured generators}
For $g \ge 0$ and $p \ge 1$, the mapping class group $\Mod(\Sigma_{g,p})$ is generated by a finite number of Dehn twists and elementary half-twists $H_1, \dots, H_{p-1}$.
\end{theorem}

\subsection{Generators for the Pure Mapping Class Group}

Generating the pure mapping class group $\text{PMod}(\Sigma_{g,p})$ requires understanding how to generate transformations that fix the $p$ punctures.
Unlike the full mapping class group, Dehn twists alone are often sufficient, provided that they are chosen to properly interact with the topology created by the punctures.
The generating sets differ based on whether there is a single puncture or multiple punctures.
\subsubsection{The Case of a Single Puncture}
When $p=1$, any homeomorphism must fix the single puncture, so the distinction between the full and pure mapping class groups vanishes: 
$$\Mod(\Sigma_{g,1}) = \text{PMod}(\Sigma_{g,1}).$$ 
The structure in this case is analogous to that of a surface with one boundary component.
A foundational result, building on Johnson's work on the Torelli group, provides a remarkably simple generating set.
\begin{theorem} \cite{johnson}
\label{thm:johnson}
Let $\Sigma_{g,1}$ be an orientable surface of genus $g \ge 1$ with one puncture.
The mapping class group $\Mod(\Sigma_{g,1})$ is generated by the $2g+1$ Dehn twists in the standard Humphries generating set for the closed surface $\Mod(\Sigma_g)$.
\end{theorem}

\subsubsection{The Case of Multiple Punctures}
When $p \ge 2$, additional generators are required to capture the interactions between punctures.
A standard finite generating set for $\text{PMod}(\Sigma_{g,p})$ can be constructed by combining generators for the underlying closed surface with twists on curves that enclose pairs of punctures.
A standard finite generating set for $\text{PMod}(\Sigma_{g,p})$ can be constructed by combining generators for the underlying closed surface with twists on curves that enclose pairs of punctures.
To avoid notational conflict with the Humphries generators (which use $a_i, b_i, c_i$), we will denote the puncture-enclosing curves by $e_i$.
\begin{theorem}\cite[Chapter~4]{farb-margalit}
\label{thm:pmod_gen}
For $g \ge 0$ and $p \ge 2$, the pure mapping class group $\text{PMod}(\Sigma_{g,p})$ is generated by the union of two sets of Dehn twists:
\begin{enumerate}
    \item[(i)] A standard generating set of Dehn twists for $\Mod(\Sigma_g)$, such as the $2g+1$ Humphries generators, viewed on $\Sigma_{g,p}$ as being supported on curves disjoint from the punctures.
    \item[(ii)] A set of $p-1$ Dehn twists about simple closed curves $\{e_1, e_2, \dots, e_{p-1}\}$, where each curve $e_i$ encloses only the punctures $z_i$ and $z_{i+1}$.
\end{enumerate}
\end{theorem}

\begin{figure}[h!]

\centering
\includegraphics[width=0.8\textwidth]{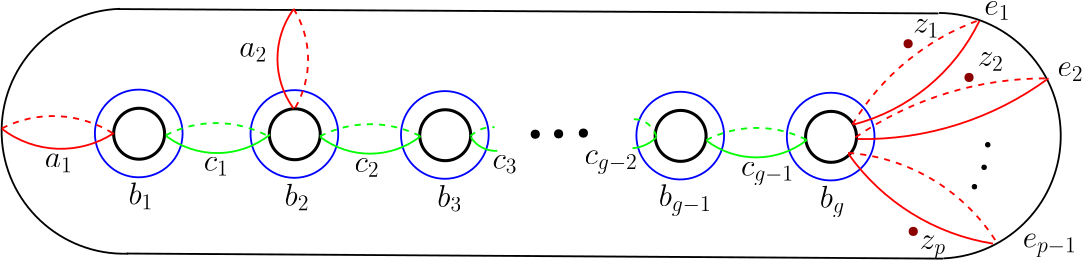}
\caption{A generating set for $\text{PMod}(\Sigma_{g,p})$. It includes the standard Humphries generators  for the genus part of the surface, plus twists about curves like that enclose pairs of punctures.}
\label{fig:pmod-generators}
\end{figure}

The generating set for $\text{PMod}(\Sigma_{g,p})$ can be refined.
For instance, it is known that not all $p-1$ puncture-pair twists are needed if one makes use of the lantern relation.
It remains an open question to determine the minimal number of generators, particularly when they are restricted to being pairwise conjugate or torsion elements of small order.
A significant result in this area is due to Monden, who in $2011$ proved the following theorem.
\begin{theorem}~\cite{Monden11}
For $g \ge 1$ and $p \ge 2$, the mapping class group $\Mod(\Sigma_{g, p})$ is generated by three elements.
\end{theorem}

Monden's work addressed the case of multiple punctures.
Monden's proof  builds upon the well-established generating sets consisting of Dehn twists.
The general strategy is to start with a known, often larger, set of generators and show that they can be expressed as products of a smaller, carefully chosen set of three elements.
The proof involves intricate manipulations of these generators and relies on the lantern relation and other known relations within the mapping class group.
The key is to select three specific mapping classes whose combinations can produce a standard set of generators, such as the Humphries generators \cite{humphries}.
The question of whether the number of generators could be further reduced remained open for some time.
It is a known fact that for most genera, the mapping class group has a non-abelian finite quotient, which implies it cannot be generated by a single element.
Thus, the theoretical minimum for a non-cyclic group is two.
In a more recent work, Monden  has improved upon the previous result.
It has now been established that:

\begin{theorem}\cite{Monden24}
For $g \ge 3$ and $p \ge 0$, the mapping class group $\Mod(\Sigma_{g, p})$ is generated by two elements.
\end{theorem}
This result is indeed minimal, as the mapping class group is not cyclic.
For surfaces with punctures, the question of finding the minimal number of involution generators is also a key problem.
Recent work by Altunöz, Pamuk, and Yildiz \cite{altunoz-pamuk-yildiz2023} provides specific bounds on the number of involution generators for the mapping class group of a punctured orientable surface $\Sigma_{g,p}$.
Their results demonstrate that for sufficiently high genus and a certain number of punctures, the mapping class group can be generated by a small number of involutions.
\begin{theorem}  \cite{altunoz-pamuk-yildiz2023}\label{thm_2023_tay}
Let $\Sigma_{g,p}$ be a connected orientable surface of genus $g$ with $p$ punctures.
The mapping class group $\mathrm{Mod}(\Sigma_{g,p})$ is generated by:
\begin{itemize}
    \item[(i)] three involutions if $g \ge 14$ and $p \ge 10$ is even.
    \item[(ii)] four involutions if $g \ge 13$ and $p \ge 9$ is odd.
    \item[(iii)] four involutions if $3 \le g \le 6$ and $p \ge 4$ is even.
\end{itemize}
\end{theorem}

These results are significant as they provide concrete numbers of involution generators for mapping class groups of punctured surfaces, contributing to the broader understanding of their algebraic structure  and the result in the first case is sharp since the mapping class group is not generated by two involutions.
The proofs of these statements often involve intricate geometric constructions of the involutions and detailed analysis of their products, showing that they can be combined to produce a known set of generators for the mapping class group, such as a set of Dehn twists.  In this chapter, we improve upon our prior results by demonstrating that Theorem~\ref{thm_2023_tay} (i) also holds for $g=13$.  The proof that follows not only establishes this improvement but also lays out the same method used to prove the original Theorem~\ref{thm_2023_tay}.
We consider embeddings of the surface $\Sigma_{g,p}$ in $\mathbb{R}^3$ that remain invariant under the $\pi$ rotations $\rho_1$ and $\rho_2$ about the z-axis (shown in ~\cite[Figure 1]{altunoz-pamuk-yildiz2023} for $g=2k$ and Figure~\ref{goddpeven_tay} for $g=2k+1$).
The mapping class group $\Mod(\Sigma_{g,p})$ contains the element $R = \rho_1\rho_2$ with the following properties:
\begin{itemize}
\item [(i)]   $R(a_i)=a_{i+1}$, $R(b_i)=b_{i+1}$ for $i=1,\ldots,g-1$ and $R(b_g)=b_{1}$,
\item [(ii)]  $R(c_i)=c_{i+1}$ for $i=1,\ldots,g-2$,
\item [(iii)] $R(z_1)=z_p$ and $R(z_i)=z_{i-1}$ for $i=2,\ldots,p$.
\end{itemize}
\begin{figure}[h]
\begin{center}
\scalebox{0.25}{\includegraphics{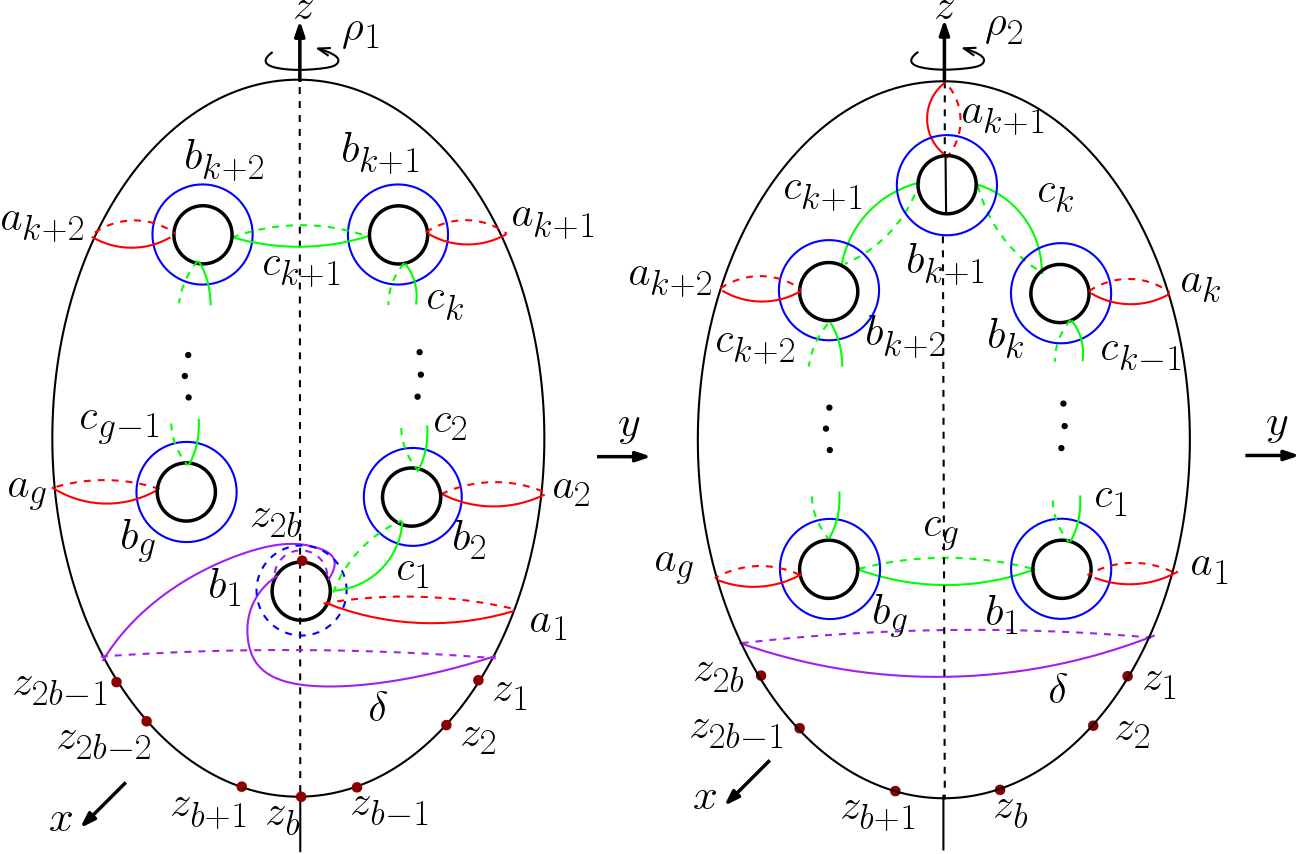}}
\caption{The involutions $\rho_1$ and $\rho_2$ if $g=2k+1$ and $p=2b$.}
\label{goddpeven_tay}
\end{center}
\end{figure}
Let $H_{i,j}$ denote the half-twist supported on the twice-punctured disk enclosing the punctures $z_i$ and $z_j$ on $N_{g,p}$.
Note that $RH_{i,j}R^{-1}=H_{i-1,j-1}$ by the action of $R$ on the punctures.
\begin{lemma}\label{lemeven}
For $g=13$ and for every even integer $p=2b\geq 8$, the subgroup of $\Mod(\Sigma_{g,p})$ generated by the elements 
\[
\rho_1, \rho_2 \textrm{ and }\rho_1H_{b-1,b}H_{b+1,b}^{-1}t_{c_3}t_{b_{6}}t_{a_{7}}
t_{a_{8}}^{-1}t_{b_{9}}^{-1}t_{c_{11}}^{-1}\]
contains the Dehn twists $t_{a_{i}}$, $t_{b_{i}}$ and $t_{c_{i}}$  for $i=1,\ldots,13$.
\end{lemma}
\begin{proof}
Consider the models of $\Sigma_{g,p}$ shown in Figure~\ref{goddpeven_tay}. Let $G$ be the subgroup of $\Mod(\Sigma_{g,p})$ generated by
\(
\rho_1,  \rho_2, \text{ and } \rho_1 G_1,
\)
where $$
G_1 := H_{b-1,b} H_{b+1,b}^{-1} t_{c_3} t_{b_6} t_{a_7} t_{a_8}^{-1} t_{b_9}^{-1} t_{c_{11}}^{-1}.
$$
One can observe that $G$ contains both $R = \rho_1 \rho_2$ and $G_1 = \rho_1 (\rho_1 G_1)$.
Let $G_2$ be the element obtained by the conjugation of $G_1$ by $R^{-2}$.
Since 
\[
R^{-2}(c_{3}, b_{6}, a_7, a_{8}, b_{9}, c_{11}) = (c_{1}, b_{4}, a_5, a_{6}, b_{7}, c_{9})
\]
and
\[
R^{-2}(z_{b-1},z_{b},z_{b+1})=(z_{b+1},z_{b+2},z_{b+3}),
\]
we have
\[G_2=R^{-2}G_{1}R^2=H_{b+1,b+2} H_{b+3,b+2}^{-1} t_{c_1} t_{b_4} t_{a_5} t_{a_6}^{-1} t_{b_7}^{-1} t_{c_{9}}^{-1}\in G. 
\]
\noindent 
Similarly, consider the element $R^3G_2R^{-3}$ contained in $G$, denoted by $G_3$:
$$
G_3=R^3G_2R^{-3}=H_{b-2,b-1} H_{b,b-1}^{-1} t_{c_4} t_{b_7} t_{a_8} t_{a_9}^{-1} t_{b_{10}}^{-1} t_{c_{12}}^{-1}.
$$
Now, let us consider the element $G_4=(G_2G_3)G_2(G_2G_3)^{-1}$.
The diffeomorphism $G_2G_3$ maps the curves $(c_{1},b_{4},a_5,a_{6},b_{7},c_{9})$ to $(c_{1},c_{4},a_5,a_{6},b_{7},b_{10})$, respectively.
The punctures $z_i$ are invariant under the action of $G_2G_3$.
Hence, we get
\begin{eqnarray*}
G_4&=&(G_2G_3)G_2(G_2G_3)^{-1}\\
&=&(G_2G_3)(H_{b+1,b+2} H_{b+3,b+2}^{-1} t_{c_1} t_{b_4} t_{a_5} t_{a_6}^{-1} t_{b_7}^{-1} t_{c_{9}}^{-1})(G_2G_3)^{-1}\\
&=&H_{b+1,b+2} H_{b+3,b+2}^{-1} t_{c_1} t_{c_4} t_{a_5} t_{a_6}^{-1} t_{b_7}^{-1} t_{b_{10}}^{-1}\in G.
\end{eqnarray*}
Then, similarly  we obtain the following two elements contained in $G$:
\begin{eqnarray*}
G_5&=&R^3G_4R^{-3}=H_{b-2,b-1} H_{b,b-1}^{-1} t_{c_4} t_{c_7} t_{a_8} t_{a_9}^{-1} t_{b_{10}}^{-1} t_{b_{13}}^{-1}\textrm{ and }\\
G_6&=&(G_4G_5^{-1})G_4(G_4G_5^{-1})^{-1}=H_{b+1,b+2} H_{b+3,b+2}^{-1} t_{c_1} t_{c_4} t_{a_5} t_{a_6}^{-1} t_{c_7}^{-1} t_{b_{10}}^{-1}.
\end{eqnarray*}
Thus, we conclude that $G_6G_5^{-1}=t_{b_7}t_{c_{7}}^{-1}\in G$. Conjugation by $R$ shows that the subgroup $G$ includes $t_{b_i}t_{c_i}^{-1}$ for each $i=1,\ldots,12$.
\noindent
Moreover, it follows from $\rho_2(b_7,c_7)=(b_8,c_5)$ that the subgroup $G$ contains the element 
\[
t_{b_8}t_{c_{7}}^{-1}=\rho_2(t_{b_7}t_{c_{7}}^{-1})\rho_2.
\]
Again, using the action of $R$, one can obtain the elements $t_{b_{i+1}}t_{c_{i}}^{-1}\in G$ for $i=1,\ldots,12$.
Consider the following elements:
\begin{eqnarray*}
G_7&=&(t_{b_{2}}t_{c_{1}}^{-1}))G_2t_{c_{9}}t_{b_{10}}^{-1}=H_{b+1,b+2} H_{b+3,b+2}^{-1} t_{b_2} t_{b_4} t_{a_5} t_{a_6}^{-1} t_{b_7}^{-1} t_{c_{10}}^{-1},\\
G_8&=&R^3G_7R^{-3}=H_{b-2,b-1} H_{b,b-1}^{-1} t_{b_5} t_{b_8} t_{a_8} t_{a_9}^{-1} t_{b_{10}}^{-1} t_{c_{13}}^{-1} \textrm{ and }\\
G_9&=&(G_7G_8)G_7(G_7G_8)^{-1}=H_{b+1,b+2} H_{b+3,b+2}^{-1} t_{b_2} t_{b_4} t_{b_5} t_{a_6}^{-1} t_{b_7}^{-1} t_{c_{10}}^{-1}.
\end{eqnarray*}
Thus, \( G_7G_9^{-1} = t_{a_5}t_{b_5}^{-1} \in G \). By the action of \( R \), it follows that the elements \( t_{a_i}t_{b_i}^{-1} \) belong to \( G \) for all \( i = 1, \dots, 13 \).
The subgroup $G$ contains the following elements:
\begin{eqnarray*}
t_{b_1}t_{b_{2}}^{-1}&=&(t_{b_1}t_{c_{1}}^{-1})(t_{c_1}t_{b_{2}}^{-1}),\\
t_{c_1}t_{c_{2}}^{-1}&=&(t_{c_1}t_{b_{2}}^{-1})(t_{b_2}t_{c_{2}}^{-1}) \textrm{ and }\\
t_{a_1}t_{b_{2}}^{-1}&=&(t_{a_1}t_{b_{1}}^{-1})(t_{b_1}t_{b_{2}}^{-1})(t_{b_2}t_{c_{2}}^{-1}).\end{eqnarray*}
By \cite[Theorem 5]{korkmaz2020}, the Dehn twists $t_{a_i}$, $t_{b_i}$, and $t_{c_i}$ lie in the subgroup generated by
\(
R, t_{a_1}t_{a_2}^{-1}, t_{b_1}t_{b_2}^{-1} \textrm{ and } t_{c_1}t_{c_2}^{-1},
\)
which concludes the proof.
\end{proof}

Using the same steps of the proof of  ~\cite[Lemma 3.10]{altunoz-pamuk-yildiz2023}, one can show that the pure mapping class group $\text{PMod}(\Sigma_{g, p})$ is contained in the subgroup $G$.
The element $\rho_1\rho_2 \in G$ maps to the $p$-cycle $(1,2,\ldots,p) \in S_p$.
Since $G$ contains all Dehn twists $t_{a_i}, t_{b_i}, t_{c_i}$ (as shown above), it follows that the half-twist composition $H_{b-1,b}H_{b+1,b}^{-1}$ also lies in $G$.
Conjugating this element with $R$, we get $H_{1,2}H_{3,2}^{-1} \in G$, whose image in $S_p$ is the $3$-cycle $(1,2,3)$.
The cycles $(1,2,\ldots,p)$ and $(1,2,3)$ generate $S_p$ when $p$ is even, by \cite[Theorem B]{iz}.
Therefore, the subgroup $G$ is mapped surjectively onto $S_p$. Finally, we conclude that Theorem~\ref{thm_2023_tay} (i) holds true even when $g=13$ and $p\geq 8$ and even.
\section{Surfaces with Boundary}
\subsection{Mapping Class Groups of Simple Surfaces with Boundary}

\begin{example}[The Disk, $\Sigma_0^1$]
The mapping class group of a disk $\Sigma_0^1$ with its boundary fixed pointwise is trivial:
$$ \Mod (\Sigma_0^1) = \{ \id \}. $$
This is a classical result known as the Alexander Lemma\index{Alexander Lemma}.
Any orientation-preserving diffeomorphism of a disk that fixes the boundary pointwise is isotopic to the identity relative to the boundary.
\end{example}

\begin{example}[The Annulus, $\Sigma_0^2$]
Let $A = \Sigma_0^2$ be an annulus\index{annulus} with two boundary components $\partial_1 A$ and $\partial_2 A$, both fixed pointwise.
The mapping class group is isomorphic to the integers:
$$ \Mod(A) \cong \mathbb{Z}. $$
This group is generated by a single Dehn twist $t_c$ about the core curve $c$ of the annulus.
Powers of $t_c$ correspond to twisting one boundary component relative to the other by multiples of $2\pi$.
\begin{figure}[h]
\begin{center}
\scalebox{0.25}{\includegraphics{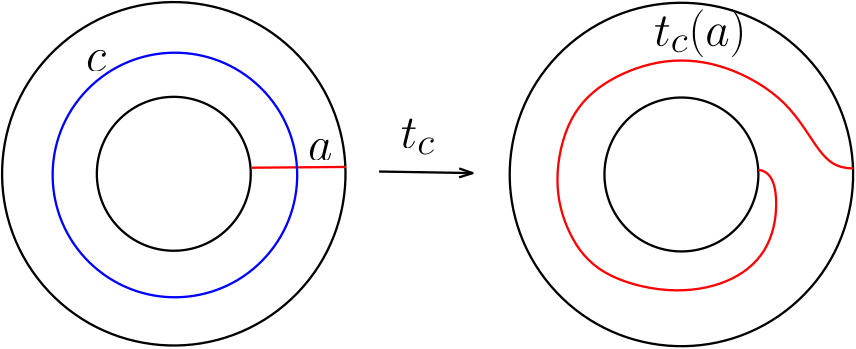}}
\caption{Generator for $\Mod(\Sigma_0^2)$.}
\label{fig:annulus}
\end{center}
\end{figure}
\end{example}

%\begin{example}[The Pair of Pants, %$\Sigma_0^3$]
%Let $P = \Sigma_0^3$ be a pair of pants\index{pair of pants} (a sphere with three disjoint open disks removed).
%If all three boundary components are fixed pointwise, the mapping class group is trivial:
%$$ \Mod(\Sigma_0^3) = \{ \id \}.
%$$
%Any essential simple closed curve in the interior of a pair of pants is parallel to one of the boundary components.
%A Dehn twist about such a curve, while keeping the boundary fixed, would necessarily be supported in an annulus bounded by this curve and the corresponding boundary component.
%Such a twist is isotopic to the identity if the boundary itself is fixed pointwise.
%The situation is analogous to twisting an annulus while holding one of its boundary circles perfectly still, an action which can always be unwound.
%\end{example}

\begin{example}[The Pair of Pants, $\Sigma_0^3$]
Let $P = \Sigma_0^3$ be a pair of pants\index{pair of pants} (a sphere with three disjoint open disks removed).  The mapping class group $\Mod(P)$ is nontrivial: 
$$ \Mod(P)\cong \mathbb{Z}^3. $$
This group is generated by Dehn twists  about the three boundary-parallel curves which clearly commute.  Hence there is an injective homomorphism $\mathbb{Z}^3 \to \Mod(P)$.  By applying the Birman exact sequence to each boundary circle yields a surjection $\Mod(P) \to \Mod(A) \cong \mathbb{Z}$  whose kernel is generated by the Dehn twist about the capped boundary, we see that the subgroup generated by the three boundary-parallel Dehn twists equals 
$\Mod(P)$.

\end{example}

\subsection{Generators for General Surfaces with Boundary}

For a general connected orientable surface\index{surface!with boundary} $\Sigma_g^b$, the mapping class group $\Mod(\Sigma_g^b)$ is finitely generated by Dehn twists.
\begin{theorem}~\cite{korkmaz2023}
For $g\geq 2$, the mapping class group $\Mod(\Sigma_g^b)$ is generated by a finite number of Dehn twists about nonseparating simple closed curves
\[
\{a_1, \dots, a_g, \, b_1, \dots, b_g, \, c_1, \dots, c_{g-1}, \, e_1, \dots, e_{b-1}\}.
\]
\end{theorem}
The curves in this generating set are depicted in Figure~\ref{orientable_boundary_generators}.
\begin{figure}[h]
\begin{center}
\scalebox{0.25}{\includegraphics{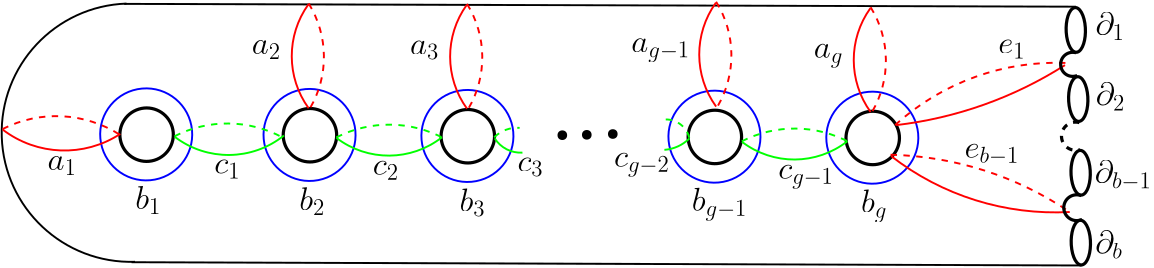}}
\caption{The curves on the surface  $\Sigma_g^b$.}
\label{orientable_boundary_generators}
\end{center}
\end{figure}
For $g=1$ and $b\geq 2$, the Dehn twists about the nonseparating curves  above cannot generate $\Mod(\Sigma_g^b)$.
The Dehn twists about the boundary parallel curves are needed to generate the whole group.
The case of surfaces with a single boundary component is closely related to the case of punctured surfaces.
\begin{theorem}~\cite{johnson}
The mapping class group $\Mod(\Sigma^1_{g})$ of an orientable surface of genus $g$ with one boundary component is generated by $2g+1$ Dehn twists.
\end{theorem}

While this result provides a finite generating set, the question of the \emph{minimal} number of generators (not necessarily Dehn twists) is a different one.
Similar to the case of punctured surfaces, it has been shown that $\Mod(\Sigma^b_{g})$ can also be generated by a smaller number of elements.
In~\cite{Wajnryb99},
Wajnryb provided an explicit finite presentation for the mapping class group $\Mod(\Sigma_{g,b})$ for any $g$ and $b$.
The minimal number of generators for $\Mod(\Sigma^b_{g})$ is known to grow with the genus and the number of boundary components.
While a single, simple formula for the minimal number of generators for all $g$ and $b$ is not as straightforward as for punctured surfaces, the following principle holds:

For the general case of $b \ge 1$ boundary components, the situation is more complex.
The mapping class group $\Mod(\Sigma^b_{g})$ can be understood through its relation to the mapping class groups of surfaces with fewer boundary components or punctures via fiber sequences.
A key result, again by Wajnryb~\cite{Wajnryb99}, provides an explicit finite presentation for the mapping class group $\Mod(\Sigma^b_{g})$ for any $g$ and $b$.
For a fixed genus $g$, the complexity of the mapping class group, and hence the number of generators in a minimal generating set, generally increases with the number of boundary components.
For many cases, especially when $g$ is large enough, it is known that two generators suffice.
However, for low genus and a high number of boundary components, the situation can be different.
The precise minimal number of generators for arbitrary $g$ and $b$ is a subject of ongoing research.
For many practical purposes, the presentations given by Wajnryb and others provide the most concrete understanding of these groups.
The general idea behind finding generating sets for surfaces with boundaries often involves relating the group to the mapping class group of the closed surface obtained by capping the boundary components with disks (or punctured disks).
The generators are then typically a combination of generators for the closed surface and Dehn twists around the boundary curves.
\subsection{Relationship with Mapping Class Groups of Closed Surfaces}
There are important relationships between mapping class groups of surfaces with boundary and those of closed surfaces, typically via capping operations.
\begin{enumerate}
    \item \text{Capping with a Disk}: If we glue a disk $D$ to a boundary component $\partial_i S$ of $\Sigma^b_g$ to obtain $\Sigma_g^{b-1}$ (or $\Sigma_g$ if $b=1$), there is a surjective homomorphism:
    $$ \text{cap}_i : \Mod(\Sigma^b_g) \to \Mod(\Sigma_g^{b-1}) $$
    The kernel of this map is generated by the Dehn twist $t_{c_i}$ about a curve $c_i$ parallel to the boundary component $\partial_i S$ that was capped.
This leads to a short exact sequence (for $b \ge 1$ and assuming $\Sigma_g^{b-1}$ is not $\Sigma^0_0$ or $\Sigma^1_0$ which have trivial $\Mod$):
    $$ 1 \to \langle t_{c_i} \rangle \cong \mathbb{Z} \to \Mod\Sigma({g, b}) \to \Mod(\Sigma_{g, b-1}) \to 1. $$
    This sequence can be used inductively.
Repeatedly capping all $b$ boundary components relates $\Mod(\Sigma_g^b)$ to $\Mod(\Sigma_g)$.
The kernel in this case would be more complex, involving multiple boundary twists.
    \item \text{Forgetting Punctures (if $n>0$):}  The Birman exact sequence relates $\Mod(\Sigma_{g, n}^b)$ to $\Mod(\Sigma_{g, n-1}^b)$:
    $$ 1 \to \pi_1(\Sigma^{b}_{g, n-1}, p_n) \xrightarrow{\text{push}} \Mod(\Sigma_{g, n}^{b}) \to \Mod(\Sigma_{g, n-1}^{b}) \to 1. $$
    The generators for the kernel are point-pushing maps, expressible as Dehn twists.
\end{enumerate}

\section{Concluding Remarks}

The century-long quest to understand the algebraic structure of the mapping class group of an orientable surface, $\Mod(\Sigma_g)$, has been a story of remarkable progress, moving from foundational geometric constructions to profound algebraic economy.
The journey began with Dehn's seminal work proving that $\Mod(\Sigma_g)$ is finitely generated, albeit by a large collection of Dehn twists \cite{dehn}.
This was followed by a classical period of refinement, led by Lickorish and Humphries, which culminated in establishing that a minimal set of $2g+1$ Dehn twists is sufficient for a surface of genus $g$ \cite{lickorish,humphries}.
The landscape of this problem shifted dramatically with Wajnryb's landmark result that this intricate group can be generated by just two elements, revealing an unexpected algebraic compactness \cite{wajnryb}.
This discovery opened a new chapter focused on ultimate minimality, a result later extended by Monden to surfaces with any number of punctures \cite{Monden24}.
Subsequent research sought to uncover the geometric nature of these minimal generating sets.
Korkmaz demonstrated that the group could be generated by a geometrically intuitive pair: a single Dehn twist and a finite-order torsion element \cite{korkmaz2005}.
The pursuit of generators with the simplest possible structure culminated in the study of involutions, leading to Korkmaz's proof that three involutions are sufficient to generate the group for high-genus surfaces \cite{korkmaz2020}.
These developments do more than just simplify the group's algebraic description.
They provide explicit, geometrically meaningful generators that are foundational to modern applications, from advancing our understanding of Teichmüller theory and its quotient, the moduli space, to enabling more efficient constructions and computations in $3$-manifold topology.
\subsection{Future Directions}
The progress in understanding minimal generation opens several promising directions for future research:

\begin{itemize}
    \item \text{Minimality in low-genus and boundary cases:} While bounds for involution generators in high-genus cases are now settled \cite{korkmaz2020}, the exact minimal number of generators for surfaces of small genus or those with multiple boundary components remains an active area of investigation.
    \item \text{Relations and presentations:} Beyond finding generators, the search for uniform, minimal presentations, which also include the relations between generators, is a critical next step.
Such presentations could significantly streamline algorithms in computational topology.

    \item \text{Dynamics and growth:} Understanding how these minimal generating sets act on key geometric spaces, such as the curve complex and Teichmüller space, may shed new light on asymptotic group invariants, growth functions, and the large-scale geometry of $\Mod(\Sigma_g)$.
    \item \text{Extensions to other contexts:} Generalizing these findings to non-orientable surfaces and decorated surfaces, for instance by considering Pin structures or cohomology with twisted coefficients, promises to uncover new torsion and involution phenomena and deepen our grasp of mapping class groups in their broadest context.
\end{itemize}

These lines of inquiry will continue to illuminate the rich interplay between mapping class groups and the wider landscapes of topology, geometry, and mathematical physics, carrying forward a century of foundational work.

\textbf{Acknowledgements}
The authors thank Athanase Papadopoulos for inviting them to write this survey and Allen Hatcher for correcting a mistake in an earlier version of this work.

\printindex

\end{document}